\documentclass[a4paper]{article}
\usepackage{fullpage}
\usepackage{amsmath}
\usepackage{amssymb}
\usepackage{amsthm}
\usepackage{graphicx}
\usepackage{bbm}
\usepackage{enumerate}
\usepackage{microtype}
\usepackage{xcolor}
\usepackage[none]{hyphenat}
\usepackage[colorlinks=true,urlcolor=blue,linkcolor=blue,plainpages=false,pdfpagelabels]{hyperref}
\allowdisplaybreaks

\newtheorem{theorem}{Theorem}[section]
\newtheorem{proposition}[theorem]{Proposition}
\newtheorem{lemma}[theorem]{Lemma}
\newtheorem{corollary}[theorem]{Corollary}
\newtheorem{remark}[theorem]{Remark}
\newtheorem{notation}[theorem]{Notation}

\newcommand{\N}{\mathbb{N}}

\newcommand{\R}{\mathbb{R}}
\newcommand{\eps}{\varepsilon}

\title{Bounds on strong unicity for Chebyshev approximation with bounded coefficients}
\author{Andrei Sipo\c s${}^{a,b}$\\[2mm]
\footnotesize ${}^a$Research Center for Logic, Optimization and Security (LOS), Department of Computer Science,\\
\footnotesize Faculty of Mathematics and Computer Science, University of Bucharest,\\
\footnotesize Academiei 14, 010014 Bucharest, Romania\\[1mm]
\footnotesize ${}^b$Simion Stoilow Institute of Mathematics of the Romanian Academy,\\
\footnotesize Calea Grivi\c tei 21, 010702 Bucharest, Romania\\[2mm]
\footnotesize E-mail: andrei.sipos@fmi.unibuc.ro\\
}
\date{}

\begin{document}

\maketitle

\begin{abstract}
We obtain new effective results in best approximation theory, specifically moduli of uniqueness and constants of strong unicity, for the problem of best uniform approximation with bounded coefficients, as first considered by Roulier and Taylor. We make use of techniques from the field of proof mining, as introduced by Kohlenbach in the 1990s. In addition, some bounds are obtained via the Lagrangian interpolation formula as extended through the use of Schur polynomials to cover the case when certain coefficients are restricted to be zero.

\noindent {\em Mathematics Subject Classification 2010}: 41A10, 41A25, 41A50, 41A52.

\noindent {\em Keywords:} Chebyshev approximation, Schur polynomials, Young tableaux, modulus of uniqueness, strong uniqueness, proof mining.
\end{abstract}

\section{Introduction}\label{sec:intro}

A typical kind of result which comes up in approximation theory is the uniqueness of the best approximation of a function taken from a generally large class (such as the class of continuous or integrable functions) towards a reasonably well-behaved object such as a polynomial or a piecewise linear function. In that vein, one may cite the classical uniqueness theorem for uniform Chebyshev approximation. We shall denote in the sequel the supremum norm by $\|\cdot\|$ and -- for any $n \in \N$ -- the class of real polynomials of degree at most $n$ by $P_n$. Then, the theorem states that for any continuous $f: [0,1] \to \R$ and any $n \in \N$, there is a unique $p \in P_n$ such that
$$\|f-p\|=\min_{q \in P_n} \|f-q\|.$$

In his 1990 PhD thesis \cite{Koh90}, Ulrich Kohlenbach carried out a program of applying techniques from proof theory, a subfield of mathematical logic, to proofs of theorems such as the above in order to compute explicit so-called `moduli of uniqueness', i.e., roughly speaking, functions $\Psi$ such that for any $f$ and $n$ as in the above, any $\eps > 0$ and any $p_1$, $p_2 \in P_n$ that have the corresponding `approximation errors' $\|f-p_1\|$ and $\|f-p_2\|$ within $\Psi(f,n,\eps)$ of the desired minimum, one can then be sure that $\|p_1 -p_2\| \leq \eps$. Such a modulus would help, for example, in calibrating the number of steps an algorithm should be run in order to obtain a polynomial as close as desired to the optimal one. The rationale for applying proof-theoretic ideas to these kinds of problems arose from a program of Georg Kreisel from the 1950s called `unwinding of proofs', which aimed at using proof transformations in order to extract new information out of potentially non-constructive proofs in ordinary mathematics. This program was later given maturity by Kohlenbach and his students and collaborators, under the name of `proof mining', and yielded several new results not only in approximation theory, but also in fields such as nonlinear analysis, ergodic theory, convex optimization or commutative algebra. A comprehensive monograph which reflects the state of the art of the field as of 2008 is \cite{Koh08}, while more recent surveys are \cite{Koh17,Koh19}.

Kohlenbach analysed initially two proofs of the uniqueness of the best Chebyshev approximation, the standard one of de la Vall\'ee Poussin \cite{VP19} and a lesser known one due to Young \cite{You07}. The latter one, although conceptually more involved, has the advantage that its analysis is much simpler due to the fact that the result which is the essential ingredient in both proofs, the alternation theorem, is used in such a way that its proof may be effectively bypassed. These two analyses were published in \cite{Koh93a, Koh93b}. Later, in a paper from 2003 \cite{KohOli03}, Kohlenbach and his student Paulo Oliva also obtained moduli of uniqueness for a result whose potential had been foreshadowed in \cite{Koh90,Koh93b}, namely best polynomial approximation with respect to the $L_1$ norm. Detailed expositions of this work may be found in \cite{Koh96} and in \cite[Chapter 16]{Koh08}, and we shall frequently reference the latter of those in the course of presenting our results.

Another case study which had been mentioned in \cite[Chapter 8]{Koh90} as a promising avenue for future research is the uniqueness of the best Chebyshev (uniform) approximation by polynomials of bounded degree with some constraints on the coefficients. This was first established in 1971 by the following result of Roulier and Taylor \cite{RouTay71}.

\begin{theorem}[{cf. \cite[Theorem 5]{RouTay71}}]\label{rt-thm}
Let $n$, $m \in \N$ be such that $m \leq n$ and $(k_i)_{i=1}^m \subseteq \N$ be such that $0 < k_1 < \ldots < k_m \leq n$. In addition, let $(a_i)_{i=1}^m$, $(b_i)_{i=1}^m \subseteq \R \cup \{\pm\infty\}$ be such that for all $i \in \{1,\ldots,m\}$, $a_i \leq b_i$, $a_i \neq \infty$ and $b_i \neq -\infty$. If one sets
$$K:=\left\{ \sum_{i=0}^n c_iX^i \in P_n \mid \ \text{for all }i \in \{1,\ldots,m\},\ a_i \leq c_{k_i} \leq b_i\right\},$$
then for any continuous $f: [0,1] \to \R$ there is a unique $p \in K$ such that
$$\|f-p\|=\min_{q \in K} \|f-q\|.$$
\end{theorem}

The goal of this paper is to obtain a modulus of uniqueness for this case of Chebyshev approximation. The main novelty is the application of Schur polynomials to obtain useful explicit formulas for the interpolation results which are needed in the proof. These formulas, together with some ways they can be bounded, are presented in Section~\ref{sec:schur}. Afterwards, in Section~\ref{sec:main}, we present the actual derivation of our modulus of uniqueness, which is followed in Section~\ref{sec:rem} by its immediate byproducts: the associated constant of strong unicity and a modulus that does not depend on a lower bound on the distance to the best approximant.

We shall use the convention $0^0=1$.

\section{Interpolation with Schur polynomials}\label{sec:schur}

A tool which is used in ordinary Chebyshev approximation and which was employed in its proof analyses is the Lagrangian interpolation formula, which we shall now review. Let $n \in \N$ and consider $n+1$ distinct points $x_1,\ldots,x_{n+1} \in [0,1]$ and $p \in P_n$. If one puts, for any $j \in \{1,\ldots,n+1\}$ and any $x \in [0,1]$,
$$l_j(x;x_1,\ldots,x_{n+1}):=\prod_{i \neq j} \frac{ x-x_i}{x_j-x_i},$$
then, for any $x \in [0,1]$,
$$p(x)=\sum_{j=1}^{n+1} l_j(x;x_1,\ldots,x_{n+1})\cdot p(x_j).$$

In their proof of Theorem~\ref{rt-thm}, Roulier and Taylor used some general interpolation results that yield polynomials where some of the coefficients are constrained to be zero -- specifically, the following two results.

\begin{theorem}[{cf. \cite[Theorem 2]{RouTay71}}]\label{rt-th2}
Let $n$, $l \in \N$ be such that $l \leq n$ and $(g_i)_{i=1}^l \subseteq \N$ be such that $0 < g_1 < \ldots < g_l \leq n$. In addition, consider $(x_j)_{j=1}^{n+1-l} \subseteq [0,1]$ such that $x_1<\ldots<x_{n+1-l}$ and let $(\alpha_j)_{j=1}^{n+1-l} \subseteq \R$. Then there exists a unique $p = \sum_{i=0}^n c_iX^i \in P_n$ such that:
\begin{itemize}
\item for all $i \in \{1,\ldots,l\}$, $c_{g_i}=0$;
\item for all $j \in \{1,\ldots,n+1-l\}$, $p(x_j)=\alpha_j$.
\end{itemize}
\end{theorem}

\begin{proposition}[{cf. \cite[Corollary 1]{RouTay71}}]\label{rt-cor1}
Let $n$, $l \in \N$ be such that $l \leq n$ and $(g_i)_{i=1}^l \subseteq \N$ be such that $0 < g_1 < \ldots < g_l \leq n$. In addition, consider $(x_j)_{j=1}^{n-l} \subseteq (0,1)$ such that $x_1<\ldots<x_{n-l}$. Then there exists a $p = \sum_{i=0}^n c_iX^i \in P_n$ such that:
\begin{itemize}
\item for all $i \in \{1,\ldots,l\}$, $c_{g_i}=0$;
\item for all $j \in \{1,\ldots,n-l\}$, $p(x_j)=0$;
\item setting $x_0:=0$ and $x_{n+1-l}:=1$, for all $j \in \{0,\ldots,n-l\}$, $p$ is nonzero on the interval $(x_j,x_{j+1})$ with sign $(-1)^{j}$.
\end{itemize}
\end{proposition}

What we want is to derive a useful explicit formula for the interpolation polynomial. For that, we presuppose the truth of Theorem~\ref{rt-th2}. Assume that we have $n$, $r \in \N$ with $r \leq n$, $(d_i)_{i=1}^{r+1} \subseteq \N$ with $n \geq d_1 > d_2 > \ldots > d_{r+1}$, $(x_j)_{j=1}^{r+1} \subseteq (0,1)$ with $x_1<\ldots<x_{r+1}$ and $(\alpha_j)_{j=1}^{r+1} \subseteq \R$. Suppose that we already have a polynomial
$$p = \sum_{i=1}^{r+1} \eta_i X^{d_i}$$
such that for all $j \in \{1,\ldots,r+1\}$,
$$p(x_j)=\alpha_j,$$
so, for all $j \in \{1,\ldots,r+1\}$,
$$\sum_{i=1}^{r+1} \eta_i x_j^{d_i} = \alpha_j.$$
Therefore one has that
$$
\begin{pmatrix}
p \\
\alpha_1 \\
\vdots \\
\alpha_{r+1}
\end{pmatrix}
=\sum_{i=1}^{r+1} \eta_i
\begin{pmatrix}
X^{d_i} \\
x_1^{d_i} \\
\vdots \\
x_{r+1}^{d_i}
\end{pmatrix},
$$
so
\begin{equation}\label{p-form}
\begin{vmatrix}
p & X^{d_1} & \cdots & X^{d_{r+1}} \\
\alpha_1 & x_1^{d_1} & \cdots & x_1^{d_{r+1}} \\
\vdots & \vdots & \ddots & \vdots \\
\alpha_{r+1} & x_{r+1}^{d_1} & \cdots & x_{r+1}^{d_{r+1}}
\end{vmatrix}
= 0.
\end{equation}

The form of the matrix above resembles a bit the one of Vandermonde determinants -- while the ordinary Vandermonde determinant, for any $r \in \N$ and any $y_1,\ldots,y_{r+1}$, is given by 
$$
\begin{vmatrix}
y_1^r & y_1^{r-1} & \cdots & 1 \\
y_2^r & y_2^{r-1} & \cdots & 1 \\
\vdots & \vdots & \ddots & \vdots \\
y_{r+1}^r & y_{r+1}^{r-1} & \cdots & 1
\end{vmatrix}
= \prod_{1\leq i<j\leq {r+1}} (y_i-y_j),$$
and we shall denote it by $V(y_1,\ldots,y_{r+1})$, the arbitrariness of the degrees in the matrix of \eqref{p-form} leads one to use the notion of a {\it generalized} Vandermonde determinant, for which one considers in addition a finite sequence $(h_i)_{i=1}^{r+1} \subseteq \N$ with $h_1 > \ldots > h_{r+1}$ and then sets
$$
V(h_1,\ldots,h_{r+1};y_1,\ldots,y_{r+1}):=
\begin{vmatrix}
y_1^{h_1} & y_1^{h_2} & \cdots & y_1^{h_{r+1}} \\
y_2^{h_1} & y_2^{h_2} & \cdots & y_2^{h_{r+1}} \\
\vdots & \vdots & \ddots & \vdots \\
y_{r+1}^{h_1} & y_{r+1}^{h_2} & \cdots & y_{r+1}^{h_{r+1}}
\end{vmatrix}.
$$
Armed with these notations, by expanding the determinant in \eqref{p-form} along its first column, we get that
$$p = \sum_{j=1}^{r+1} (-1)^{j-1} \frac{V(d_1,\ldots,d_{r+1};X,x_1,\ldots,\widehat{x_j},\ldots,x_{r+1})}{V(d_1,\ldots,d_{r+1};x_1,\ldots,x_{r+1})} \cdot\alpha_j.$$

In order to obtain a workable formula for $p$, we shall make use of some concepts and results of algebraic combinatorics. The standard reference for these notions is \cite[Part I]{Mac95}. By a {\it partition}, in the following, we shall mean a finite sequence $(\lambda_i)_{i=1}^{r+1} \subseteq \N$ with $\lambda_1 \geq \ldots \geq \lambda_{r+1}$.  To any finite sequence $h = (h_i)_{i=1}^{r+1} \subseteq \N$ with $h_1 > \ldots > h_{r+1}$ as before, one associates a partition $\lambda^h$ by putting, for any $i \in \{1,\ldots, {r+1}\}$, $\lambda^h_i:=h_i+i-r-1$. (It is easy to check that this correspondence is actually bijective.) To any partition one can in turn associate a multivariate polynomial by the following procedure. If $r \in \N$ and $\lambda$ is a partition of length $r+1$, then a {\it semistandard Young tableau} of weight $\lambda$ is a jagged array with $r+1$ rows where for any $i \in \{1,\ldots,{r+1}\}$, the $i$'th line has $\lambda_i$ entries which are elements of the set $\{1,\ldots,{r+1}\}$, such that the entries on each row are (weakly) increasing and the entries on each column are strictly increasing. If $T$ is such a semistandard Young tableau in which for each $i \in \{1,\ldots,{r+1}\}$, $i$ appears $t_i$ times in $T$, one denotes by $y^T$ the monomial $y_1^{t_1}\ldots y_{r+1}^{t_{r+1}}$. Then the {\it Schur polynomial} associated to $\lambda$ is defined by
$$s_\lambda:= \sum_{T} y^T,$$
where $T$ ranges over all semistandard Young tableaux of weight $\lambda$. One may easily show that this polynomial is symmetric.

The relevant result here (a simple proof may be found in \cite{Pro89}) states that for any $r$, any strictly decreasing sequence $h$ of length ${r+1}$ and any $y_1,\ldots,y_{r+1}$,
$$V(h_1,\ldots,h_{r+1};y_1,\ldots,y_{r+1}) = V(y_1,\ldots,y_{r+1}) \cdot s_{\lambda^h}(y_1,\ldots,y_{r+1}).$$

The formula above for $p$ now becomes
$$p = \sum_{j=1}^{r+1} (-1)^{j-1} \frac{V(X,x_1,\ldots,\widehat{x_j},\ldots,x_{r+1}) \cdot s_{\lambda^d}(X,x_1,\ldots,\widehat{x_j},\ldots,x_{r+1})}{V(x_1,\ldots,x_{r+1}) \cdot s_{\lambda^d}(x_1,\ldots,x_{r+1})} \cdot\alpha_j.$$

Since, for any $j$,
$$(-1)^{j-1} \frac{V(X,x_1,\ldots,\widehat{x_j},\ldots,x_{r+1})}{V(x_1,\ldots,x_{r+1})} = l_j(X;x_1,\ldots,x_{r+1}),$$
we have that
\begin{equation}\label{schf}
p = \sum_{j=1}^{r+1} l_j(X;x_1,\ldots,x_{r+1}) \cdot\alpha_j \cdot \frac{s_{\lambda^d}(X,x_1,\ldots,\widehat{x_j},\ldots,x_{r+1})}{s_{\lambda^d}(x_1,\ldots,x_{r+1})},
\end{equation}
a formula that differs from the Lagrangian one only by the additional Schur factors.

For any partition $\lambda$ of length ${r+1}$, the number of semistandard Young tableaux of weight $\lambda$ can be shown to be
$$N_\lambda:=\prod_{1\leq i<j\leq {r+1}} \frac{\lambda_i-\lambda_j+j-i}{j-i}.$$
Moreover, for any $n$ there is a finite number of strictly decreasing sequences $h$ with length smaller than or equal to $n+1$ and with $h_1 \leq n$. If we set, for any $n$, $N_n$ to be the maximum of all the $N_{\lambda^h}$'s for all these $h$'s, this number is easily seen to be computable. The following bound is now immediate.

\begin{proposition}\label{s-ub}
For all $n$, $r \in \N$ with $r \leq n$, any strictly decreasing sequence $h$ of length ${r+1}$ and with $h_1\leq n$, and any $y_1,\ldots,y_{r+1} \in [0,1]$,
$$0 \leq s_{\lambda^h}(y_1,\ldots,y_{r+1}) \leq N_n.$$
\end{proposition}

In order to obtain meaningful (i.e. nonzero) {\it lower} bounds on the Schur polynomials, we must capitalize on the hypotheses of our problem. Theorem~\ref{rt-th2} above only concerns cases where the degrees of the coefficients which are required to be zero -- the $g_i$'s -- are nonzero, so the degrees of the coefficients which form the parameters of our problem -- the $d_i$'s -- contain $0$, and since the sequence $(d_i)$ is decreasing, we have that $d_{r+1}=0$. Therefore it makes sense to focus on this kind of strictly decreasing sequences.

\begin{proposition}\label{s-lb1}
Let $n$, $r \in \N$ with $r \leq n$, and let $h$ be a strictly decreasing sequence of length ${r+1}$ with $h_1 \leq n$ and $h_{r+1}=0$. Let $\delta > 0$ and $y_1,\ldots,y_{r+1} \in [0,1]$ be such that at most one of the $y_i$'s is strictly smaller than $\delta$. Then
$$s_{\lambda^h}(y_1,\ldots,y_{r+1}) \geq \delta^{\frac{n^2}4}.$$
\end{proposition}

\begin{proof}
Since $s_{\lambda^h}$ is symmetric, we may assume that for all $j \in \{2,\ldots,{r+1}\}$, $y_j \geq \delta$.

Also, by the definition of $\lambda^h$, we have that $\lambda^h_{r+1}=0$. Therefore, we may construct a semistandard Young tableau of weight $\lambda^h$ in the following way: one fills, for each $i \in \{1,\ldots,r\}$, each entry in the $i$'th row with the number $i+1$. Since all $y_i$'s are nonnegative, the monomials associated with the other possible tableaux will be nonnegative, and therefore one obtains the lower bound
$$s_{\lambda^h}(y_1,\ldots,y_{r+1}) \geq y_2^{\lambda^h_1}\ldots y_{r+1}^{\lambda^h_{r}} \geq \delta^{\sum_{i=1}^{r} \lambda^h_i}.$$
On the other hand, one has
$$\sum_{i=1}^{r} \lambda^h_i \leq r \cdot \lambda^h_1 = r(h_1 - r) \leq r(n-r) \leq \frac{n^2}4,$$
which finishes our proof.
\end{proof}

In particular, the proposition above guarantees that the formula for $p$ is valid i.e. the Schur denominator is nonzero (as one can take, for example, $\delta$ such that $x_1<\delta<x_2$).

We shall also need the following particular kind of lower bound.

\begin{proposition}\label{s-lb2}
Let $n$, $r \in \N$ with $r \leq n$, and let $h$ be a strictly decreasing sequence of length ${r+1}$ with $h_1 \leq n$ and $h_{r+1}=0$. Let $\alpha > 0$ and $y_1,\ldots,y_{r} \in [0,1]$ be such that for all $j \in \{1,\ldots,r-1\}$, $y_{j+1}-y_j \geq \alpha$. Set
$$L:=\left\{y \in [0,1] \mid \text{for all }j \in \{1,\ldots,r\},\ |y_j - y| \geq \alpha \right\}.$$
Then, for all $y \in L$,
$$s_{\lambda^h}(y,y_1,\ldots,y_{r}) \geq \alpha^{\frac{n^2}4}.$$
\end{proposition}

\begin{proof}
Let $y \in L$. We must show that the numbers $y$, $y_1,\ldots,y_{r}$ fulfill the condition of Proposition~\ref{s-lb1} with $\alpha$ playing the role of $\delta$. Clearly, for all $j \in \{2,\ldots,r\}$, $y_j \geq \alpha$. Assume that $y_1 < \alpha$. Then, since $y \in L$, we cannot have $y < y_1$. Therefore $y - y_1 = |y_1 - y| \geq \alpha$, so $y \geq y_1 + \alpha \geq \alpha$.
\end{proof}

Consider now the case treated in Proposition~\ref{rt-cor1}. In our setting, what we do is to set $x_{r+1}:=1$ and for all $j \in \{1,\ldots,r\}$, $\alpha_j:=0$. Then the formula for $p$ becomes
$$p = l_{r+1}(X;x_1,\ldots,x_{r},1) \cdot\alpha_{r+1} \cdot \frac{s_{\lambda^d}(X,x_1,\ldots,x_{r})}{s_{\lambda^d}(x_1,\ldots,x_{r},1)},$$
and -- by suitably setting $\alpha_{r+1}$ -- we get
$$p=\prod_{i=1}^{r} (x_i-X) \cdot s_{\lambda^d}(X,x_1,\ldots,x_{r}).$$
Since in this case $x_1 > 0$, we may apply Proposition~\ref{s-lb1} for a $\delta \in (0,x_1)$ to obtain that the Schur factor is always strictly positive, and therefore the additional sign information given by Proposition~\ref{rt-cor1} immediately follows. We have thus derived in the process Proposition~\ref{rt-cor1} as a corollary of Theorem~\ref{rt-th2}.

\section{Main results}\label{sec:main}

The following two lemmas wrap up the results of the previous section and yield some bounds which are useful for our particular problem.

\begin{lemma}\label{beta-lemma}
Let $n$, $r \in \N$ with $r \leq n$ and $(d_i)_{i=1}^{r+1} \subseteq \N$ with $n \geq d_1 > d_2 > \ldots > d_{r+1} =0$. Let $\beta$, $\gamma > 0$ with $\beta \leq 1$ and $(x_j)_{j=1}^{r+1} \subseteq [0,1]$ be such that for all $j \in \{1,\ldots,r\}$, $x_{j+1}-x_j \geq \beta$. Suppose that we have a polynomial
$$p = \sum_{i=1}^{r+1} \eta_i X^{d_i}$$
such that for all $j \in \{1,\ldots,{r+1}\}$,
$$|p(x_j)|\leq \frac{\beta^{n+ \frac{n^2}4}}{N_n \cdot (n+1)} \cdot \gamma.$$
Then $\|p\| \leq \gamma$.
\end{lemma}

\begin{proof}
Let $x \in [0,1]$. By the formula \eqref{schf}, we have that
$$p(x) = \sum_{j=1}^{r+1} l_j(x;x_1,\ldots,x_{r+1}) \cdot p(x_j) \cdot \frac{s_{\lambda^d}(x,x_1,\ldots,\widehat{x_j},\ldots,x_{r+1})}{s_{\lambda^d}(x_1,\ldots,x_{r+1})}.$$
Clearly, we have, using that $\beta \leq 1$,
$$|l_j(x;x_1,\ldots,x_{r+1})| \leq \frac1{\prod_{i \neq j} \beta |i-j|} \leq \frac1{\beta^r} \leq \frac1{\beta^n}.$$
By Propositions~\ref{s-ub} and \ref{s-lb1},
$$\left|\frac{s_{\lambda^d}(x,x_1,\ldots,\widehat{x_j},\ldots,x_{r+1})}{s_{\lambda^d}(x_1,\ldots,x_{r+1})} \right| \leq \frac{N_n}{\beta^{\frac{n^2}4}}.$$
Therefore
\begin{align*}
|p(x)| &\leq \sum_{j=1}^{r+1} |l_j(x;x_1,\ldots,x_{r+1})| \cdot |p(x_j)| \cdot \left|\frac{s_{\lambda^d}(x,x_1,\ldots,\widehat{x_j},\ldots,x_{r+1})}{s_{\lambda^d}(x_1,\ldots,x_{r+1})}\right| \\
&\leq (n+1) \cdot \frac1{\beta^n} \cdot \frac{\beta^{n+ \frac{n^2}4}}{N_n \cdot (n+1)} \cdot \gamma \cdot \frac{N_n}{\beta^{\frac{n^2}4}} = \gamma,
\end{align*}
and we are done.
\end{proof}

\begin{lemma}\label{p-lb}
Let $n$, $r \in \N$ with $r \leq n$, and let $h$ be a strictly decreasing sequence of length ${r+1}$ with $h_1 \leq n$ and $h_{r+1}=0$. Let $\alpha \in (0,1]$ and $z_1,\ldots,z_{r} \in [0,1]$ such that for all $j \in \{1,\ldots,r-1\}$, $z_{j+1}-z_j \geq \alpha$. Set
$$L:=\left\{x \in [0,1] \mid \text{for all }j \in \{1,\ldots,r\},\ |z_j - x| \geq \alpha  \right\}$$
and
$$p:=\prod_{j=1}^{r} (z_j-X) s_{\lambda^h}(X,z_1,\ldots,z_{r}).$$
Then, for all $x \in L$,
$$|p(x)| \geq \alpha^{\frac{n^2}4+n}.$$
\end{lemma}

\begin{proof}
Let $x \in L$. Since $\alpha \leq 1$,
$$\left|\prod_{j=1}^{r} (z_j-x) \right| \geq \prod_{j=1}^{r} |z_j-x| \geq \alpha^r \geq \alpha^n.$$
The result follows by Proposition~\ref{s-lb2}.
\end{proof}

We shall need in the sequel the following notion: a {\it modulus of uniform continuity} for a function $f: [0,1] \to \R$ is a function $\omega : (0, \infty) \to (0,\infty)$ such that for any $\eps >0$ and any $x$, $y \in [0,1]$ with $|x-y| < \omega(\eps)$, we have that $|f(x)-f(y)| < \eps$. Clearly, a function $f: [0,1] \to \R$ has a modulus of uniform continuity if and only if it is uniformly continuous.

\begin{notation}
Let $\omega : (0, \infty) \to (0,\infty)$, $n \in \N$ and $M \geq 0$. We shall set, for any $\eps>0$,
$$\chi_{\omega,n,M}(\eps):=\min \left( 1, \frac{\eps}{4n^2M+1}, \omega\left(\frac\eps2\right)\right).$$
\end{notation}

In addition, we shall need the following classical inequality.

\begin{lemma}[Markov brothers' inequality]\label{mark}
Let $q \in P_n$. Then
$$\max_{x \in [-1,1]} |q'(x)| \leq n^2 \cdot \max_{x\in[-1,1]} |q(x)|.$$
\end{lemma}

\begin{corollary}\label{mark2}
Let $p \in P_n$. Then
$$\max_{x \in [0,1]} |p'(x)| \leq 2n^2 \cdot \max_{x\in[0,1]} |p(x)|.$$
\end{corollary}

\begin{proof}
Let $q:=p\left(\frac{X+1}2\right)$. Then $q \in P_n$ and $q'=\frac12 \cdot p'\left(\frac{X+1}2\right)$. Using Lemma~\ref{mark}, we see that
\begin{align*}
\max_{x \in [0,1]} |p'(x)| &=\max_{x \in [-1,1]} \left|p'\left(\frac{x+1}2\right)\right| = 2 \cdot \max_{x \in [-1,1]} |q'(x)| \\
&\leq  2n^2 \cdot \max_{x\in[-1,1]} |q(x)| = 2n^2 \cdot \max_{x\in[-1,1]} \left|p\left(\frac{x+1}2\right)\right| = 2n^2 \cdot \max_{x\in[0,1]} |p(x)|.
\end{align*}
\end{proof}

\begin{corollary}\label{mark3}
Let $p \in P_n$ and $k \in \mathbb{N}$. Then
$$\max_{x \in [0,1]} |p^{(k)}(x)| \leq (2n^2)^k \cdot \max_{x\in[0,1]} |p(x)|.$$
In addition, if $a_k$ is the $k$'th coefficient of $p$, then
$$|a_k| \leq \frac{(2n^2)^k}{k!} \max_{x\in[0,1]} |p(x)|.$$
\end{corollary}

\begin{proof}
The first statement follows easily from Corollary~\ref{mark2}, by induction on $k$. For the second statement, we use the fact that $p^{(k)}(0)=k! \cdot a_k$.
\end{proof}

\begin{proposition}[{cf. \cite[p. 318]{Koh08}}]\label{chi-prop}
Let $\omega : (0, \infty) \to (0,\infty)$, $n \in \N$ and $M \geq 0$. Let $p \in P_n$ with $\|p\| \leq M$ and $f:[0,1] \to \R$ be such that $\omega$ is a modulus of uniform continuity for $f$. Then $\chi_{\omega,n,M}$ is a modulus of uniform continuity for $p-f$.
\end{proposition}

\begin{proof}
Let $\eps >0$ and $x$, $y \in [0,1]$ with $|x-y| < \chi_{\omega,n,M}(\eps)$. By Corollary~\ref{mark2}, we have that $\|p'\| \leq 2n^2\|p\| \leq 2n^2M$. Applying now the mean value theorem, we get that there is a $c \in (x,y) \subseteq (0,1)$ such that
$$|p(x)-p(y)| = |p'(c)||x-y| \leq 2n^2M|x-y| < \frac{2n^2M\eps}{4n^2M+1} < \frac\eps2.$$
In addition, since $|x-y| < \omega\left(\frac\eps2\right)$, we also have that $|f(x) - f(y)| <\frac\eps2$, from which the conclusion follows.
\end{proof}

\begin{notation}
For all $n \geq 0$, put
$$F_n:=\frac{3}2\max_{i \in \{0,\ldots,n\}} \frac{(2n^2)^{i}}{i!} \geq 1.$$
\end{notation}

The following theorem is an analogue of \cite[Theorem 16.26]{Koh08} and may be considered to be a generalized (`approximate') version of the corresponding alternation theorem for this setting, i.e. the case where $0<k_0$ of \cite[Theorem 3]{RouTay71}, which we recover by setting $\eps:=0$.

\begin{theorem}\label{alt-mess}
Let $n$, $m \in \N$ be such that $m \leq n$ and $(k_i)_{i=1}^m \subseteq \N$ be such that $0 < k_1 < \ldots < k_m \leq n$. In addition, let $(a_i)_{i=1}^m$, $(b_i)_{i=1}^m \subseteq \R \cup \{\pm\infty\}$ be such that for all $i \in \{1,\ldots,m\}$, $a_i \leq b_i$, $a_i \neq \infty$ and $b_i \neq -\infty$. Set
$$K:=\left\{ \sum_{i=0}^n c_iX^i \in P_n \mid \ \text{for all }i \in \{1,\ldots,m\},\ a_i \leq c_{k_i} \leq b_i\right\}.$$
Let $p_0 \in K$. Let $f: [0,1] \to \R$ and let $\omega : (0, \infty) \to (0,\infty)$ be a modulus of uniform continuity for $f$. Set $M:=\frac52\|f\| + \frac32\|p_0\|$ and
$$E:=\min_{q \in K} \|f-q\|.$$
Let $\eps \in \left[0, \frac{E}4\right)$, $L \in (0,E]$ and $p \in K$ such that $\|p\|\leq M$ and
$$\|f-p\| \leq E +  \frac{\left(\frac{\chi_{\omega,n,M}\left(\frac{L}2\right)}2\right)^{\frac{n^2}4+n}}{N_n} \cdot\eps.$$
Set $\mu:= F_n \cdot\eps$. Let $(c_j)_{j=0}^n$ be the coefficients of $p$. Put $l \leq m$ and $(e_v)_{v=1}^l$ -- uniquely determined! -- such that:
\begin{enumerate}[(i)]
\item $1 \leq e_1 < \ldots < e_l \leq m$;
\item for all $v \in \{1,\ldots,l\}$, $c_{k_{e_v}} \leq a_{e_v} + \mu$ or $c_{k_{e_v}} \geq b_{e_v} - \mu$;
\item for all $i \in \{1,\ldots,m\} \setminus \{e_1,\ldots,e_l\}$, $a_i + \mu < c_{k_i} < b_i - \mu$.
\end{enumerate}
Then there is a finite sequence $(x_j)_{j=1}^{n+1-l} \subseteq [0,1]$ with $x_1<\ldots<x_{n+1-l}$ and there is a $\nu \in \{\pm1\}$ such that for all $i \in \{1,\ldots,n+1-l\}$,
$$|\nu(-1)^i(f(x_i)-p(x_i)) - E| \leq \eps.$$
\end{theorem}

\begin{proof}
Put, for each $v \in \{1,\ldots,l\}$, $g_v:=k_{e_v}$.

Since $\|p\| \leq M$, by Proposition~\ref{chi-prop}, $\chi_{\omega,n,M}$ is a modulus of uniform continuity for $p-f$. We shall write in the remainder of the proof $\chi$ instead of $\chi_{\omega,n,M}$.

Divide the interval $[0,1]$ into subintervals $I_1,\ldots,I_u$ of length $\chi\left(\frac{L}2\right)$, except for the last one which may be shorter. The amplitude of $p-f$ on each such interval is less than $\frac{L}2$, so it is less than $\frac{E}2$. Among those intervals, we distinguish {\it special intervals} as being those intervals which contain a point $x$ with $E - \eps \leq |p(x)-f(x)| \leq E + \eps$. Since $\eps < \frac{E}2$, the function $p-f$ is nonzero -- with constant sign -- on each special interval. We therefore classify special intervals into positive and negative intervals, and if we conceive of their enumeration to consist of successive groups of positive and negative intervals, our goal is to show that the number of these groups is at least $n+1-l$. Assume without loss of generality that the first special interval is positive.

Since $\chi\left(\frac{L}2\right) \leq 1$ (by its definition), we have that 
$$\|f-p\| \leq E +  \frac{\left(\frac{\chi\left(\frac{L}2\right)}2\right)^{\frac{n^2}4+n}}{N_n} \cdot\eps \leq E + \eps.$$

\ \\[2mm]\noindent {\bf Claim} (cf. \cite[Lemmas 16.5 and 16.25]{Koh08}){\bf .} The number $w$ of special interval groups is at least $2$.\\[1mm]

\noindent {\bf Proof of the claim:} We have to show that there is an $x$ such that $p(x)-f(x) \leq -E+\eps$ and an $x$ such that $E-\eps \leq p(x)-f(x)$. We shall show only the existence of the first kind of $x$, the existence of the second kind following similarly. Assume towards a contradiction that for all $x\in [0,1]$,
$$p(x)-f(x) > -E+\eps,$$
so
$$\min_{x\in [0,1]} (p(x)-f(x)) > -E+\eps.$$

Set $h:=\frac12\left(\min_{x\in [0,1]} (p(x)-f(x))+E\right)+\frac\eps2$. Note that $\eps < h \leq E+\eps$ and that
$$\min_{x\in [0,1]} (p(x)-f(x))  = -E-\eps+2h,$$
so for all $x\in [0,1]$,
$$-E -\eps + 2h \leq p(x)-f(x) \leq E+\eps,$$
that is, by subtracting $h$,
$$-E-\eps+h \leq p(x)-h-f(x) \leq E+\eps-h.$$
Now remark that $0\leq E+\eps-h<E$, so if we put $q:=p-h$, we have that $q \in K$ (as $0 < k_1$) and that $\|f-q\| \leq E+\eps-h < E$. On the other hand, we know that
$$E=\min_{q \in K} \|f-q\|,$$
so we have a contradiction.\hfill $\blacksquare$\\[2mm]

Now, assume towards a contradiction that $w<n+1-l$. By the constant sign property, we may select between each two successive groups one non-special interval. Take $z_1,\ldots,z_{w-1}$ to be the midpoints of these selected non-special intervals.

By our assumption, we may choose $(d_i)_{i=1}^w \subset \N$ such that $n \geq d_1 > \ldots > d_w = 0$ and $\{d_1,\ldots,d_w\} \subseteq \{0,\ldots,n\} \setminus \{g_1,\ldots,g_l\}$.

Set
$$\rho:=\prod_{i=1}^{w-1} (z_i-X) \cdot s_{\lambda^d}(X,z_1,\ldots,z_{w-1}).$$

By the discussion at the end of the previous section, we have that the degree of $\rho$ is less than or equal to $n$, and the coefficients of degree $g_1,\ldots, g_l$ are zero. In addition, on each special interval, $\rho$ is nonzero and has the same sign as $p-f$. Clearly $\|\rho\| \leq N_n$ and by the definition of the $z_i$'s we have that for all $j \in \{1,\ldots,w-2\}$,
$$z_{j+1}-z_j \geq 2\chi\left(\frac{L}2\right) \geq \frac{\chi\left(\frac{L}2\right)}2.$$
Put
$$L:=\left\{x \in [0,1] \mid \text{for all }i \in \{1,\ldots,r\},\ |z_j - x| \geq \frac{\chi\left(\frac{L}2\right)}2 \right\}.$$
Then, again by the definition of the $z_i$'s, we have that all special intervals are contained in $L$. By Lemma~\ref{p-lb}, we have that for any $x$ in a special interval,
\begin{equation}\label{sinv}
|\rho(x)|\geq \left(\frac{\chi\left(\frac{L}2\right)}2\right)^{\frac{n^2}4+n}.
\end{equation}
Let $E^*$ be the maximum taken over all $x$ in non-special intervals of $|p(x)-f(x)|$, which is strictly smaller than $E-\eps$. Since, in addition, $\eps < \frac E4$ and $\|\rho\| \leq N_n$, one may choose $\lambda > 0$ with $\lambda \leq \frac\eps{2N_n}$, $\lambda\|\rho\| < E - E^* - \eps$ and $\lambda\|\rho\| \leq \frac E4 - \frac{\eps}{N_n}\|\rho\|$.

Put $Q:=p - \left(\lambda + \frac\eps{N_n} \right) \cdot \rho$. Clearly, $Q \in P_n$. We want to show that $Q \in K$. Let $(c'_j)_{j=0}^n$ be the coefficients of $Q$ and let $i \in \{1,\ldots,m\}$. We must show that $a_i \leq c'_{k_i} \leq b_i$. If there is a $v$ such that $i=e_v$, then $k_i=k_{e_v}=g_v$ and therefore the $k_i$'th coefficient of $\rho$ is zero and so $c'_{k_i}=c_{k_i}$ -- the conclusion then follows because $p \in K$. If there isn't a $v$ with $i=e_v$, we have that $a_i + \mu < c_{k_i} < b_i - \mu$. Using Corollary~\ref{mark3}, we have that
$$\left(\lambda+\frac\eps{N_n}\right)|\rho_{k_i}| \leq \left(\frac\eps{2N_n}+\frac\eps{N_n}\right)|\rho_{k_i}| \leq \frac{3\eps}{2N_n} \frac{(2n^2)^{k_i}}{k_i!} \|\rho\| \leq \frac{3\eps}{2} \frac{(2n^2)^{k_i}}{k_i!} \leq \mu.$$
Then we have that
$$\mu - \left(\lambda+\frac\eps{N_n}\right)\rho_{k_i} \geq 0,$$
so
$$c'_{k_i} = c_{k_i} - \left(\lambda+\frac\eps{N_n}\right)\rho_{k_i} > a_i + \mu - \left(\lambda+\frac\eps{N_n}\right)\rho_{k_i} \geq a_i.$$
Similarly one shows $c'_{k_i} \leq b_i$.

We shall now show that for all $x \in [0,1]$, $|Q(x)-f(x)| < E$, contradicting the definition of $E$.

If $x$ is not in a special interval, then
$$|Q(x) - f(x)| \leq |p(x) - f(x)| + \left(\lambda+\frac\eps{N_n}\right)|\rho(x)| \leq E^* + \frac\eps{N_n}|\rho(x)| + \lambda|\rho(x)| < E^* + \eps + E - E^* - \eps = E.$$

If $x$ is in a special interval, then on the one hand
$$|p(x) - f(x)| \geq E - \eps - \frac{E}2 > \frac{E}4$$
and on the other hand
$$\left(\lambda+\frac\eps{N_n}\right)|\rho(x)| \leq \frac\eps{N_n}|\rho(x)| + \lambda|\rho(x)| \leq \frac{E}4.$$
Since in this case, $p(x) - f(x)$ and $\left(\lambda+\frac\eps{N_n}\right)\rho(x)$ have the same sign, one may write, using \eqref{sinv},
\begin{align*}
|Q(x) - f(x)| &= \left|p(x) - f(x) - \left(\lambda+\frac\eps{N_n}\right)\rho(x)\right| \\
&= |p(x) - f(x)| - \left|\left(\lambda+\frac\eps{N_n}\right)\rho(x)\right| \\
&\leq E + \frac{\left(\frac{\chi\left(\frac{L}2\right)}2\right)^{\frac{n^2}4+n}}{N_n} \cdot\eps - \frac\eps{N_n}|\rho(x)| - \lambda|\rho(x)| < E.
\end{align*}
The conclusion now follows.
\end{proof}

The following theorem is the main result of this paper (and it is the analogue of \cite[Theorem 16.30]{Koh08}). Similarly to the previous theorem, it implies back the ordinary uniqueness result of Theorem~\ref{rt-thm}, that is, \cite[Theorem 5]{RouTay71}.

\begin{theorem}[{effective modulus of uniqueness}]\label{modulus}
Let $n$, $m \in \N$ be such that $m \leq n$ and $(k_i)_{i=1}^m \subseteq \N$ be such that $0 < k_1 < \ldots < k_m \leq n$. In addition, let $(a_i)_{i=1}^m, (b_i)_{i=1}^m \subseteq \R \cup \{\pm\infty\}$ be such that for all $i \in \{1,\ldots,m\}$, $a_i \leq b_i$, $a_i \neq \infty$ and $b_i \neq -\infty$. Set
$$K:=\left\{ \sum_{i=0}^n c_iX^i \in P_n \mid \ \text{for all }i \in \{1,\ldots,m\},\ a_i \leq c_{k_i} \leq b_i\right\}.$$
Let $p_0 \in K$. Let $f: [0,1] \to \R$ and let $\omega : (0, \infty) \to (0,\infty)$ be a modulus of uniform continuity for $f$. Set $M:=\frac52\|f\| + \frac32\|p_0\|$ and
$$E:=\min_{q \in K} \|f-q\|.$$
Let $\delta \geq 0$, $L \in (0,E]$ and $p_1$, $p_2 \in K$ such that for each $i \in \{1,2\}$,
$$\|f-p_i\| \leq E +  \frac{\left(\frac{\chi_{\omega,n,M}\left(\frac{L}2\right)}2\right)^{\frac{n^2}2+2n}}{10 \cdot N_n^2(n+1)(nF_n+1)}\cdot\delta.$$
Then $\|p_1 - p_2\| \leq \delta$.
\end{theorem}

\begin{proof}
If $E \leq \frac25 \delta$, then, since $\|f-p_1\| \leq E + \frac1{10} \cdot\delta$ and $\|f-p_2\| \leq E + \frac1{10} \cdot\delta$,
$$\|p_1-p_2\| \leq \|f-p_1\| + \|f-p_2\| \leq 2E + \frac15 \cdot \delta \leq \delta.$$
Therefore we may assume for the rest of the proof that $E > \frac25 \delta$.

Now, assume towards a contradiction that $\|p_1\| > M = \frac52\|f\| + \frac32\|p_0\|$. Then
$$\|f-p_1\| \geq \|p_1\| - \|f\| > \frac32\|f\| + \frac32\|p_0\| \geq \frac32\|f-p_0\| \geq \frac32 E.$$
On the other hand,
$\|f-p_1\| \leq E + \frac1{10} \cdot\delta$, so $\frac32 E \leq E + \frac1{10} \cdot\delta$, which contradicts the fact that $E > \frac25 \delta$. Thus, $\|p_1\| \leq M$ and similarly $\|p_2\| \leq M$. Put $p:=\frac{p_1+p_2}2$. Then $p\in K$, $\|p\|\leq M$ and
$$\|f-p\| \leq E +  \frac{\left(\frac{\chi_{\omega,n,M}\left(\frac{L}2\right)}2\right)^{\frac{n^2}2+2n}}{10 \cdot N_n^2(n+1)(nF_n+1)}\cdot\delta.$$
Put
$$\eps:= \frac{\left(\frac{\chi_{\omega,n,M}\left(\frac{L}2\right)}2\right)^{\frac{n^2}4+n}}{10 \cdot N_n(n+1)(nF_n+1)} \cdot\delta.$$
and $\mu:=F_n \cdot\eps$.
Since $E > \frac25 \delta$, we have that $\eps \leq \frac E 4$.

Let $(c_j)_{j=0}^n$ be the coefficients of $p$. Put $l \leq m$ and $(e_v)_{v=1}^l$ -- uniquely determined! -- such that:
\begin{enumerate}[(i)]
\item $1 \leq e_1 < \ldots < e_l \leq m$;
\item for all $v \in \{1,\ldots,l\}$, $c_{k_{e_v}} \leq a_{e_v} + \mu$ or $c_{k_{e_v}} \geq b_{e_v} - \mu$;
\item for all $i \in \{1,\ldots,m\} \setminus \{e_1,\ldots,e_l\}$, $a_i + \mu < c_{k_i} < b_i - \mu$.
\end{enumerate}
Applying Theorem~\ref{alt-mess}, there is a finite sequence $(x_j)_{j=1}^{n+1-l} \subseteq [0,1]$ with $x_1<\ldots<x_{n+1-l}$ and there is a $\nu \in \{\pm1\}$ such that for all $i \in \{1,\ldots,n+1-l\}$,
$$|\nu(-1)^i(f(x_i)-p(x_i)) - E| \leq \eps \leq \frac E 4,$$
so
$$|\nu(-1)^i(f(x_i)-p(x_i)) - E| + \frac{3L}4 \leq \frac{E}4 + \frac{3E}4 = E.$$

\ \\[2mm]\noindent {\bf Claim.}  For each $i \in \{1,\ldots,n-l\}$,
$$|(p-f)(x_{i+1})-(p-f)(x_i)| \geq \frac L 2.$$

\noindent {\bf Proof of the claim:} Let $i \in \{1,\ldots,n-l\}$. Since
$$|\nu(-1)^i(f(x_i)-p(x_i)) - E| \leq E - \frac{3L}4,$$
we have that
$$-E + \frac{3L}4 \leq \nu(-1)^i(f(x_i)-p(x_i)) - E,$$
so
$$\nu(-1)^i(f(x_i)-p(x_i)) \geq \frac{3L}4 \geq 0.$$
Similarly,
$$\nu(-1)^{i+1}(f(x_{i+1})-p(x_{i+1})) \geq \frac{3L}4 \geq 0.$$
Therefore,
\begin{align*}
|(p-f)(x_{i+1})-(p-f)(x_i)| &= |\nu(-1)^i(f(x_i)-p(x_i)) + \nu(-1)^{i+1}(f(x_{i+1})-p(x_{i+1}))|\\
&= \nu(-1)^i(f(x_i)-p(x_i)) + \nu(-1)^{i+1}(f(x_{i+1})-p(x_{i+1}))\\
&\geq \frac{3L}4 + \frac{3L}4 = \frac{3L}2 \geq \frac L 2
\end{align*}
and we are done.\hfill $\blacksquare$\\[2mm]

Since $\|p\|\leq M$, $\chi_{\omega,n,M}$ is a modulus of uniform continuity for $p-f$, so for each $i \in \{1,\ldots,n-l\}$, $x_{i+1}-x_i \geq \chi_{\omega,n,M}\left(\frac{L}2\right) \geq \frac{\chi_{\omega,n,M}\left(\frac{L}2\right)}2$.

Let $Q_1$ be the polynomial obtained from $p_1-p_2$ by retaining only the terms of degrees $(k_{e_v})_{v=1}^l$ and set $Q_2:=p_1-p_2-Q_1$. It is enough to show that $\|Q_1\|\leq \frac\delta2$ and $\|Q_2\|\leq \frac\delta2$.

Let $v \in \{1,\ldots,l\}$. Then $c_{k_{e_v}} \leq a_{e_v} + \mu$ or $c_{k_{e_v}} \geq b_{e_v} - \mu$. Without loss of generality, assume $c_{k_{e_v}} \leq a_{e_v} + \mu$. (In the case where $c_{k_{e_v}} \geq b_{e_v} - \mu$, the corresponding proof of the upcoming claim mirrors the one given.) Put $c:= c_{k_{e_v}}$ and $c_1$, $c_2$, $c'$ be the $k_{e_v}$'th coefficients of $p_1$, $p_2$ and $Q_1$, respectively. Then $c=\frac{c_1+c_2}2$, $c'=c_1-c_2$, $c_1 \geq a_{e_v}$ and $c_2 \geq a_{e_v}$.

\ \\[2mm]\noindent {\bf Claim.}  We have that $|c'| \leq 2\mu$.

\noindent {\bf Proof of the claim:} Since
$$\frac{c_1+c_2}2 \leq a_{e_v} + \mu,$$
we have that
$$\frac{c_1-a_{e_v}}2 + \frac{c_2-a_{e_v}}2 \leq \mu,$$
so $c_1-a_{e_v}$ and $c_2-a_{e_v}$ are contained in the interval $[0,2\mu]$. From this we get $|c_1-c_2| \leq 2\mu$.\hfill $\blacksquare$\\[2mm]

Therefore all the coefficients of $Q_1$, which is of degree at most $n$ and has no constant term (since $0 < k_1$), are bounded in absolute value by $2\mu$, so
$$\|Q_1\| \leq n \cdot 2\mu = 2n\eps F_n \leq   \frac{2n}{10 \cdot (n F_n +1)} \cdot \delta \cdot F_n \leq \frac\delta2 .$$

Let $j \in \{1,\ldots,n+1-l\}$. Applying the argument in \cite[p. 316]{Koh08} with $q:=4\eps$, we get that $|p_1(x_j) - p_2(x_j)| \leq 4\eps$. Therefore
$$|Q_2(x_j)|\leq |p_1(x_j) - p_2(x_j)| + |Q_1(x_j)| \leq 4(nF_n+1)\eps =   \frac{4 \cdot\left(\frac{\chi_{\omega,n,M}\left(\frac{L}2\right)}2\right)^{\frac{n^2}4+n}}{10 \cdot N_n(n+1)} \cdot \delta \leq \frac{\left(\frac{\chi_{\omega,n,M}\left(\frac{L}2\right)}2\right)^{\frac{n^2}4+n}}{N_n(n+1)} \cdot \frac\delta2.$$

Since $\frac{\chi_{\omega,n,M}\left(\frac{L}2\right)}2 \leq 1$, we may apply Lemma~\ref{beta-lemma} to get that $\|Q_2\| \leq \frac\delta2$.
\end{proof}

\section{Further remarks}\label{sec:rem}

The modulus of uniqueness obtained above is in particular suited for the classical problem of Chebyshev approximation, so it is of course natural to ask how it compares with the one in \cite[Theorem 16.30]{Koh08}. The answer is that it is slightly more inefficient, but we must be careful to distinguish between inessential inefficiencies that we introduced in order to obtain a smoother exposition and those that derive from the fact that we treat the case of bounded coefficients. Among the former kind, one counts the factorial-like factors in the original modulus, whose removal makes our modulus slightly smaller than it could have been, but easier on the eyes. Among the latter kind, we mention the factor in the denominator that includes $F_n$, but most importantly the fact that the exponent of the uniform continuity modulus now includes a quadratic term which arises from the use of Proposition~\ref{s-lb1}, where one needs to bound the term $r(n-r)$ in a way that is uniform in $r$, while in the classical Chebyshev approximation where ordinary Lagrangian interpolation is used, one necessarily has $r=n$ and thus the corresponding term is $0$.

\begin{remark} \rm
It may seem surprising at a first glance that the modulus of uniqueness does not contain any dependence on the coefficient bounds except for the obvious one via $\|p_0\|$. Ulrich Kohlenbach has pointed out to the author, however, that the bounding of the norms of the $p_i$'s by $M$ (which occurs in the only case of the theorem that really matters, as clearly seen in the proof above) already restricts -- using Corollary~\ref{mark3} -- {\bf all} their coefficients to compact sets and therefore -- by the logical metatheorems in \cite[Chapter 15]{Koh08} -- no additional restrictions can contribute in any way to the final extracted quantity.
\end{remark}

\begin{remark} \rm
The modulus which we obtained depends, in addition to $\|p_0\|$, $n$ and $\delta$, on parameters specific to $f$, namely $\omega$, $L$ and $\|f\|$. One can completely eliminate the dependence on $\|f\|$ using the following trick (see \cite[p. 300]{Koh08}). By shifting the data (the $f$ and the $p_i$'s) by the constant $f(0)$ one remains within the framework given by the $K$ and the $p_0$, as the constant terms of the polynomials are not affected by the restrictions. Since the modulus of uniform continuity $\omega$ is retained, the modulus of uniqueness for this new case is then also valid for the original data, so one only has to find an upper bound for the norm of the new $f$ solely in terms of $\omega$. Let $x \in [0,1]$. Then there is an $r \leq \left\lfloor 2/{\omega(1)}\right\rfloor$ such that
$$r \cdot \frac{\omega(1)}2 \leq x < (r+1)\cdot \frac{\omega(1)}2$$
(an $r \in \N$ surely exists, and if one had $r \geq \left\lfloor 2/{\omega(1)}\right\rfloor + 1$, this would contradict $x \leq 1$). We have, then, that $0 \leq x - r \cdot \omega(1) / 2 < \omega(1)$ and since now $f(0)=0$,
$$|f(x)| = \left| f(x) - f(r \cdot \omega(1) / 2) + \sum_{i=0}^{r-1} (f((i+1)\cdot \omega(1) / 2) - f(i\cdot \omega(1) / 2)) \right| \leq 1 + r \leq \left\lfloor 2/{\omega(1)}\right\rfloor + 1.$$
\end{remark}

We notice that the modulus of uniqueness just obtained is linear in the variable $\delta$. This is connected to the property called `strong uniqueness' (or `strong unicity'), introduced in \cite[Section 3]{NewSha63}. If $X$ is a real normed space, $K \subseteq X$, $f \in X \setminus K$ and $y\in K$ such that $E:=\|f-y\| =\min_{q \in K} \|f-q\|>0$, then this property of strong uniqueness means that there is a $\gamma >0$ -- called the `constant of strong unicity' -- such that for all $p \in K$,
$$\|f-p\| \geq E+\gamma\|p-y\|.$$
Since the above can be written as 
$$\|p-y\| \leq \frac{\|f-p\| -E}\gamma,$$
it is equivalent to: for all $\delta \geq 0$,
$$\text{if }\frac{\|f-p\| -E}\gamma \leq \delta\text{, then }\|p-y\|\leq\delta,$$
or, more simply, to the fact that for all $\delta \geq 0$,
$$\text{if }\|f-p\| \leq E + \gamma\cdot\delta\text{, then }\|p-y\|\leq\delta.$$
But this last statement is provided by Theorem~\ref{modulus}, so our effective modulus of uniqueness immediately yields an effective constant of strong unicity. From a qualitative standpoint, this method may be said to provide us with an alternative route towards this strong uniqueness property, essentially different from the usual non-constructive one taken e.g. by \cite{ChaTay83,Xu91}.

If we do not care about linearity, the modulus obtained above may be improved by removing its dependence on a lower bound for $E$, as shown by the following proposition.

\begin{proposition}[{cf. \cite[Proposition 16.18]{Koh08}}]
Let $X$ be a real normed space, $K \subseteq X$ and $f \in X \setminus K$. Put
$$E:=\min_{q \in K} \|f-q\|.$$
Assume $E > 0$. Let $\Psi : [0,\infty) \times (0,\infty) \to (0, \infty)$ be such that (i) for all $\delta \geq 0$, all $L \in (0,E]$ and all $p_1$, $p_2 \in K$ such that for all $i \in \{1,2\}$, $\|f-p_i\| \leq E+\Psi(\delta,L)$, we have that $\|p_1-p_2\|\leq \delta$ and (ii) for all $L \in (0,E]$, $\Psi(0,L)=0$.

Put, for all $\delta > 0$, $\Psi^*(\delta):=\min\left(\frac\delta4, \Psi\left(\delta,\frac\delta4\right)\right)$ and $\Psi^*(0):=0$. Then, for all $\delta \geq 0$, and all $p_1$, $p_2 \in K$ such that for all $i \in \{1,2\}$, $\|f-p_i\| \leq E+\Psi^*(\delta)$, we have that $\|p_1-p_2\|\leq \delta$.
\end{proposition}

\begin{proof}
The case $\delta=0$ is obvious. If $E \geq \frac\delta4$, the conclusion is immediate, by taking $L:=\frac\delta4$. If $E< \frac\delta4$, then $\|f-p_1\| \leq E+\frac\delta4 \leq \frac\delta2$. Similarly, $\|f-p_2\|\leq\frac\delta2$, so $\|p_1-p_2\|\leq \delta$.
\end{proof}

In addition, the above moduli of uniqueness also give effective moduli of pointwise continuity and/or effective pointwise Lipschitz constants for the projection operator, as shown e.g. by \cite[Proposition 16.2]{Koh08}.

\section{Acknowledgements}

I would like to thank Ulrich Kohlenbach for his valuable remarks.

This work has been supported by the German Science Foundation (DFG Project KO 1737/6-1).

\end{document}